\documentclass[letterpaper]{amsart}
\usepackage{amsfonts,amssymb,amscd,amsmath,enumerate,url,verbatim}
\usepackage[all]{xy}
\SelectTips{eu}{}
\xyoption{curve}
\CompileMatrices
\newcommand{\ov}{\overline}

\newcommand{\col}{\colon}

\newcommand{\fm}{{\mathfrak m}}

\newcommand{\fs}{{\mathfrak s}}

\newcommand{\fn}{{\mathfrak n}}
\newcommand{\fp}{{\mathfrak p}}

\newcommand{\Ima}{\operatorname{Im}}

\newcommand{\Ker}{\operatorname{Ker}}

\newcommand{\rank}{\operatorname{rank}}

\newcommand{\Tor}{\operatorname{Tor}}
\newcommand{\Ext}{\operatorname{Ext}}

\newcommand{\Po}{\operatorname{P}}

\newcommand{\agr}[2][{}]{{{#2}^{\mathsf g}_{#1}}}

\newcommand{\Soc}{\operatorname{Soc}}

\theoremstyle{plain}
\newtheorem{theorem}{Theorem}[section]

\newtheorem*{Theorem1}{Theorem 1}
\newtheorem*{Theorem2}{Theorem 2}

\newtheorem*{Main Theorem}{Main Theorem}

\newtheorem{proposition}[theorem]{Proposition}
\newtheorem{lemma}[theorem]{Lemma}

\newtheorem{corollary}[theorem]{Corollary}
\newtheorem{subtheorem}{Theorem}[theorem]
{\Alph{theorem}}
{\Alph{theorem}}

\theoremstyle{definition}
\newtheorem*{definition}{Definition}

\newtheorem{chunk}[theorem]{}
\newtheorem{subchunk}[subtheorem]{}

\theoremstyle{remark}

\newenvironment{bfchunk}{\begin{chunk}\textit}{\end{chunk}}

\newtheorem*{Case1}{Case 1}
\newtheorem*{Case2}{Case 2}
\newtheorem*{Case3}{Case 3}
\newtheorem*{Case4}{Case 4}

\newtheorem{remark}[theorem]{Remark}

\numberwithin{equation}{theorem}

\begin{document}

\title[Conditions for the Yoneda algebra]{Conditions for the Yoneda algebra\\ of a local ring\\to be generated in low degrees}
\begin{abstract} The powers $\fm^n$ of the maximal ideal $\fm$ of a local Noetherian ring $R$ are known to satisfy certain homological properties for {\it large} values of $n$. For example, the homomorphism $R\to R/\fm^n$ is Golod for $n\gg 0$. We study when such properties hold for {\it small} values of $n$, and we make connections with  the structure of the Yoneda Ext algebra, and more precisely with the property that the Yoneda algebra of $R$ is generated in degrees $1$ and $2$. A complete treatment of these properties is pursued in the case of compressed Gorenstein local rings. 
\end{abstract}

\author[J.~Hoffmeier]{Justin Hoffmeier}
\address{Justin~Hoffmeier\\ Department of Mathematics and Statistics\\
   Northwest Missouri State University,  Maryville\\ MO 64468\\ U.S.A.}
     \email{jhoff@nwmissouri.edu}

\author[L.~M.~\c{S}ega]{Liana M.~\c{S}ega}
\address{Liana M.~\c{S}ega\\ Department of Mathematics and Statistics\\
   University of Missouri\\ \linebreak Kansas City\\ MO 64110\\ U.S.A.}
     \email{segal@umkc.edu}

\subjclass[2000]{}
\keywords{}
\thanks{Research partly supported by NSF grant DMS-1101131 and grants from the Simons Foundation (\#20903 and \#354594, Liana Sega)}

\maketitle

\section*{Introduction}

Let $(R,\fm,k)$ be a {\it local ring}, that is, a commutative noetherian ring $R$ with unique maximal ideal $\fm$ and $k=R/\fm$.  For $n\ge 1$ we let $\nu_n\col \fm^n\to \fm^{n-1}$ denote the canonical inclusion and for each $i\ge 0$ we consider the induced maps 
$$\Tor_i^R(\nu_n,k)\col\Tor_i^R(\fm^n,k)\to \Tor_i^R(\fm^{n-1},k)\,.$$
Using the terminology of \cite{small}, we say that $\fm^n$ is a {\it small} submodule of $\fm^{n-1}$ if $\Tor_i^R(\nu_n,k)=0$ for  all $i\ge 0$. This condition  implies that the canonical projection  $\rho_n\col R\to R/\fm^n$ is a Golod homomorphism, but the converse may not hold. 

Levin \cite{Levin} showed that $\fm^n$ is a small submodule of $\fm^{n-1}$ for all sufficiently large values of $n$.  On the other hand, the fact that $\fm^n$ is a small submodule of $\fm^{n-1}$ for small values of $n$ is an indicator of strong homological properties. It is known that  $\fm^2$ is a small submodule of $\fm$ if and only if the Yoneda algebra $\Ext_R(k,k)$ is generated in degree $1$, cf.\! \cite[Corollary 1]{Ro}.  More generally, we show:

\begin{Theorem1}
Let $(R,\fm,k)$ be a  local ring. Let $\widehat R=Q/I$ be a minimal Cohen presentation of $R$,  with $(Q,\fn,k)$ a regular local ring and $I\subseteq \fn^2$. Let $t$ be an integer such that $I\subseteq \fn^t$. The following statements are then equivalent: 
\begin{enumerate}[\quad\rm(1)]
\item $\fm^t$ is a small submodule of $\fm^{t-1}$; 
\item $\rho_t\col R\to R/\fm^t$ is Golod;
\item $\rho_n\col R\to R/\fm^n$ is Golod for all $n$ such that $t\le n\le 2t-2$;
\item $I\cap \fn^{t+1}\subseteq \fn I$ and the algebra $\Ext_R^*(k,k)$ is generated by $\Ext^1_R(k,k)$ and $\Ext^2_R(k,k)$.
\end{enumerate}
\end{Theorem1}

If $R$ is artinian, its {\it socle degree} is the largest integer $s$ with $\fm^s\ne 0$.  When $R$ is a compressed Gorenstein local ring (see Section 3 for a definition) of socle degree $s\ne 3$, we determine all values of the integer $n$ for which the homomorphism $\rho_n$ is Golod, respectively for which $\fm^n$ is a small submodule of $\fm^{n-1}$, and we use Theorem 1 to establish part (3) below. 

\begin{Theorem2}
Let $(R,\fm,k)$ be a compressed Gorenstein local ring of socle degree $s$. Assume $2\le s\ne 3$ and  let $t$ denote the smallest integer such that $2t\ge s+1$. 
If $n\ge 1$, then the following hold: 
\begin{enumerate}[\quad\rm(1)]
\item $\fm^n$ is a small submodule of $\fm^{n-1}$  if and only if $n>s$ or $n=s+2-t$. 
\item $\rho_n\col R\to R/\fm^n$ is Golod if and only if $n\ge s+2-t$.
\item If $s$ is even, then $\Ext_R(k,k)$ is generated by $\Ext^1_R(k,k)$ and $\Ext^2_R(k,k)$. 
\end{enumerate}
\end{Theorem2}

The conclusion of (3) does not hold when $s$ is odd, see Corollary \ref{cor}. 

Section 1 provides  definitions and properties  of the homological notions of interest. Theorem 1 is proved in Section 2 and Theorem 2 is proved in Section 3.

\section{Preliminaries}
 Throughout the paper $(R,\fm, k)$ denotes a commutative noetherian local ring with maximal ideal $\fm$ and  residue field $k$. Let  $M$ be a finitely generated $R$-module. 

We denote by $\widehat R$ the completion of $R$ with respect to $\fm$. A {\it minimal Cohen presentation} of $R$ is a presentation $\widehat R=Q/I$, with $Q$ a regular local ring with maximal ideal $\fn$ and $I$ an ideal with $I\subseteq \fn^2$. We know that such a presentation exists, by the Cohen structure theorem.

We denote by $\agr R$ the associated graded ring with respect to $\fm$, and by $\agr M$ the associated graded module with respect to $\fm$.  We denote by $(\agr R)_j$ the $j$-th graded component of $\agr R$. 
For any $x\in R$ we denote by $x^*$  the image of $x$ in $\fm^j/\fm^{j+1}=(\agr R)_j$, where $j$ is such that $x\in\fm^j\smallsetminus\fm^{j+1}$. For an ideal $J$ of $R$, we denote by $J^*$ the homogeneous ideal generated by the elements $x^*$ with  $x\in J$.

\begin{remark}\label{thom}
\label{independence}
With $\widehat R=Q/I$ as above, the following then hold:
\begin{enumerate}[\quad\rm(1)]
\item $I\subseteq \fn^t$ if and only if $\rank_k(\fm^{t-1}/\fm^{t})=\displaystyle{\binom{e+t-2}{e-1}}$, where $e$ denotes the minimal number of generators of $\fm$. 
\item Assume $t\ge 2$ and $I\subseteq \fn^t$. Then $I\cap \fn^{t+1}\subseteq \fn I$ if and only if the map 
$$\Ext^2_{\rho_t}(k,k)\colon\Ext_{R/\fm^t}^2(k,k)\to \Ext_R^2(k,k)$$ induced by the canonical projection $\rho_t\colon R\to R/\fm^t$ is surjective. 
\end{enumerate}
To prove (1), note that $\agr{(\widehat R)}=\agr Q/I^*$ and $\agr Q$ is isomorphic to a  polynomial ring over $k$ in $e$ variables of degree $1$. 
We have that $I\subseteq\fn^t$ if and only if $I^*\subseteq(\agr{\fn})^t$, which is equivalent to 
$
(\agr{Q})_{t-1}=(\agr{Q}/I^*)_{t-1}
$
and thus to 
$$
\rank_k(\agr Q)_{t-1}=\rank_k(\agr{\widehat R})_{t-1}\,.
$$
Therefore, (1) follows by noting that $\rank_k(\agr{Q})_j=\displaystyle{\binom{e-1+j}{e-1}}$ and $\rank_k(\agr{\widehat R})_j=\rank_k(\fm^j/\fm^{j+1})$ for each $j$. 

For a proof of (2), see \c Sega \cite[4.3]{defect}, noting that the map $\Ext^2_{\rho_t}(k,k)$ is surjective if and only if the induced map 
$$
\Tor_2^{\rho_t}(k,k)\colon\Tor_2^{R}(k,k)\to \Tor_2^{R/\fm^t}(k,k)
$$
is injective. 
\end{remark}

\begin{definition}
\label{t-homog}
Let $\widehat R=Q/I$ be a minimal Cohen presentation of $R$.  Let $t\ge 2$ be an integer. We say that the local ring $R$ is {\it $t$-homogeneous} if $I\subseteq \fn^t$ and $I\cap \fn^{t+1}\subseteq \fn I$. Remark \ref{independence} shows that this definition does not depend on the choice of the minimal Cohen presentation.
\end{definition}

We set
$$
v(R)=\sup\{t\ge 0\mid I\subseteq \fn^t\}\,.
$$
Note that, if $R$ is $t$-homogeneous and $I\ne 0$, then $t=v(R)$.

\begin{remark}
The terminology of {\it two-homogeneous algebra} was  previously used by L\"ofwall \cite{Lo}, with a different meaning, in the context of augmented graded algebras. 
\end{remark}

\begin{lemma}\label{degreet}
Let $\widehat R=Q/I$ be a minimal Cohen presentation of $R$. If the ideal $I^*$ of the polynomial ring $\agr Q$ is generated by homogeneous polynomials of degree $t$, then the ring $R$ is $t$-homogeneous. 
\end{lemma}

\begin{proof}
Assume $I^*$ is generated by homogeneous polynomials of degree $t$. In particular, it follows that $I\subseteq \fn^t$.  To prove $I\cap \fn^{t+1}\subseteq \fn I$, we will show $I\cap \fn^{t+1} \subseteq \fn I + \fn^j$ for all $j\gg0$. The Krull intersection theorem then gives the conclusion.

Let $x \in I\cap \fn^{t+1}$ and let $a\ge 1$ be the smallest integer such that $x\not\in \fn^{t+1+a}$. Then $x^*$ is an element of degree $t+a$ of $I^*$. Since  $I^*$ is generated by homogeneous elements of degree $t$, we can write 
$$
x^*=\sum y_i^* z^*_i
$$
with $y_i^*\in (\agr Q)_a$  and $z_i^*\in I^*\cap (\agr Q)_t$  for each $i$, where $y_i\in \fn^a$ and $z_i\in I\cap \fn^t$.  Set  
$$
x_1= x-\sum y_i z_i
$$
and note that $x_1\in I\cap \fn^{t+a+1}$ and $x-x_1\in \fn I$. In particular $x\in \fn I + \fn^{t+a+1}$. Applying the argument above to $x_1$, we obtain an element $a_1$ such that $a_1>a$, and an element  $x_2$ such that $x_2\in I\cap \fn^{t+a_1+1}$ and  $x_1-x_2\in \fn I$.  In particular $x_1$, and thus $x$, are elements of $\fn I+\fn^{t+a_1+1}$. An inductive argument  produces a sequence of integers $1<a_1<a_2<\dots$ such that $x\in \fn I + \fn^{t+a_i+1}$ for all $i$, and  gives the desired conclusion.
\end{proof}

\begin{remark}
The converse of the lemma does not hold. This can be seen by considering the $2$-homogeneous local ring $R=k[[x,y]]/(x^2+y^3, xy)$, for which $\agr R=k[x,y]/(x^2,y^4, xy)$. 
\end{remark}

We now proceed to provide definitions for the homological notions of interest, and recall some of their properties. 

The \textit{Poincar\'e series} $\Po^R_M(z)$ of $M$ is the formal power series 
$$
\Po^R_M(z)=\sum_{i\ge 0}\rank_k(\Tor_i^R(M,k))z^i.
$$
\begin{bfchunk}{Golod rings, modules, and homomorphisms.}\label{goloddef}
Let $(S, \fs, k)$ be a local ring and $\varphi\col R\to S$ be a surjective homomorphism of local rings. Following Levin \cite{Golodmodules}, we say that an $S$-module $M$ is $\varphi$-{\it Golod} if the following equality is satisfied: $$\Po_M^S(z)=\frac{\Po_M^R(z)}{(1-z(\Po_S^R(z)-1))}.$$ 

We say that $\varphi$ is a \textit{Golod homomorphism} if $k$ is a $\varphi$-Golod module.

The ring $R$ is said to be a \textit{Golod ring} if the canonical projection $Q\to \widehat{R}$ is a Golod homomorphism, where $\widehat{R}=Q/I$ is a minimal Cohen presentation.  This definition is independent of the choice of representation by \cite[Lemma 4.1.3]{Luchonotes}. Note that this definition is equivalent to the definition given by Avramov in terms of Koszul homology in \cite[\S 5]{Luchonotes}. A classical example of a Golod ring is $Q/\fn^j$ for any $j\ge 2$, where $(Q,\fn,k)$ is a regular local ring, as first observed by Golod \cite{Golod} .
\end{bfchunk}

\begin{bfchunk}{Small homomorphisms.}\label{smallstuff}
Let $\varphi: R\to S$ be a surjective ring homomorphism as above and consider the  induced maps
$$\Ext_{\varphi}^i(k,k):\Ext_S^i(k,k)\to\Ext_R^i(k,k)$$
$$\Tor_i^{\varphi}(k,k):\Tor_i^R(k,k)\to\Tor_i^S(k,k).$$

We say that $\varphi$ is \textit{small} if $\Ext_{\varphi}^*(k,k)$ is surjective or, equivalently, if $\Tor_*^{\varphi}(k,k)$ is injective.
Note that $\Tor_1^{\varphi}(k,k)$ can be identified with the canonical map 
$$\fm/\fm^2\to \fs/\fs^2\,.$$
induced by $\varphi$.  Thus $\Tor_1^{\varphi}(k,k)$ is an isomorphism if and only if $\Ker(\varphi)\subseteq \fm^2$. 

For convenience we collect below a few known facts about small homomorphisms:
\begin{enumerate}[\quad\rm(1)]
\item If $\varphi$ is Golod then $\varphi$ is small, see Avramov \cite[3.5]{small}.
\item If $S$ is a Golod ring and $\varphi$ is small, then $\varphi$ is Golod, see \c Sega \cite[6.7]{maxprops}.
\end{enumerate}
\end{bfchunk}

\begin{bfchunk}{Inert modules.}
\label{eq}
Let $\varkappa\col P\to R$ be a surjective homomorphism of local rings. Following Lescot \cite{Inert}, we say that an $R$-module $M$ is \textit{inert by $\varkappa$} if the following equality  holds: $$\Po_k^R(z) \Po_M^P(z) = \Po_k^P(z) \Po_M^R(z).$$

If $\varkappa$ is a Golod homomorphism the following are equivalent:
\begin{enumerate}[\quad\rm(1)]
\item $M$ is $\varkappa$-Golod;
\item $\Tor_i^{\varkappa}(M,k)$ is injective for all $i$;
\item $M$ is inert by $\varkappa$.
\end{enumerate}
 The equivalence of (1) and (3) is a direct consequence of the definitions and the equivalence of (1) with (2) is given by Levin \cite[1.1]{Golodmodules}.

\begin{subchunk}\label{2of3}
Consider a sequence of surjective homomorphisms of local rings 
$$R\xrightarrow{\alpha} S\xrightarrow{\beta} T\,.$$
Lescot \cite[Theorem 3.6] {Inert} shows that a $T$-module $M$ is inert by $\beta\circ \alpha$ if an only if $M$ is inert by $\beta$ and $M$ is inert by $\alpha$, when considered as an $S$-module. 
\end{subchunk}

\begin{subchunk}\label{inert3}
Let $Q$ be a regular local ring with maximal ideal $\fn$ and an ideal $J$ such that $J\subseteq\fn^t$ with $t\ge 2$. Set $S=Q/J$ and $\overline{\fn}=\fn/J$. If a finitely generated $S$-module $N$ is annihilated by $\overline{\fn}^{t-1}$ then $N$ is inert by the natural projection $\varkappa: Q\to S$,  see \cite[3.7 and 3.11]{Inert}
\end{subchunk}
\end{bfchunk}

We end this section with introducing some more notation. 
\begin{chunk}
\label{A=0}
If $(R,\fm,k)$ is a local ring and $M$ is a finite $R$-module, we let $\nu_M\col \fm M\to M$ denote the canonical inclusion and consider the induced maps $$\Tor_i^R(\nu_M,k)
\col\Tor_i^R(\fm M,k)\to\Tor_i^R(M,k)\,.$$
They fit into the long exact sequence $$\cdots\to\Tor_i^R(\fm M,k)\xrightarrow{\Tor_i^R(\nu_M,k)}\Tor_i^R(M,k)\to\Tor_i^R(M/\fm M,k)\to\cdots.$$ 

If $\Lambda$ is a graded vector space over $k$, we set  $H_{\Lambda}(z)=\sum_{i\ge 0} \rank_k(\Lambda_i)z^i$; this formal power series is called the {\it Hilbert series} of $\Lambda$.  
Since rank is additive on exact sequences, a rank count in the exact sequence above gives
\begin{equation}\label{poincare}
\Po_M^R(z)-a\Po_k^R(z)+z\Po_{\fm M}^R(z)=(1+z)H_{\Ima(\Tor_*^R(\nu_M,k))}(z)
\end{equation}
where $a=\rank_k(M/\fm M)$. 
We set
$$
T^R_M(z)=H_{\Ima(\Tor_*^R(\nu_M,k))}(z)\,.
$$
\end{chunk}

\section{Homological properties of powers of the maximal ideal}
The purpose of this section is to prove Theorem 1 in the introduction, restated as Theorem \ref{K2} below. 

For each integer $j$ let  
$$
\rho_j\col R\to R/\fm^j\qquad\text{and}\qquad \nu_j\col \fm^j\to \fm^{j-1}$$
 denote the canonical projection, respectively the inclusion, and consider the induced maps
\begin{align*}
\Tor_i^{\rho_j}(R/\fm^{j-1},k)&\colon \Tor_i^{R}(R/\fm^{j-1},k)\to \Tor_i^{R/\fm^j}(R/\fm^{j-1},k)\\
\Tor_i^R(\nu_j,k)&\col \Tor_i^R(\fm^j,k)\to \Tor_i^R(\fm^{j-1},k).
\end{align*}
Using the terminolgy in the appendix of \cite{small}, we say that $\fm^j$ is a {\it small submodule} of $\fm^{j-1}$ if  $\Tor_i^R(\nu_j,k)=0$ for all $i\ge 0$.

\begin{remark}
\label{zero}
For $i,j\ge 0$ let 
 $$\eta_{i}^j\col \Tor_{i}^R(R/\fm^j,k)\to\Tor_{i}^R(R/\fm^{j-1},k)$$
denote the map induced by the canonical projection $R/\fm^j\to R/\fm^{j-1}$. 

Note that $\Tor_i^R(\nu_j,k)=0$ if and only if $\eta_{i+1}^j=0$.  Indeed, this is a standard argument, using the canonical isomorphisms
$
\Tor_{i+1}^R(R/\fm^{n},k)\cong \Tor_{i}^R(\fm^n,k)
$
which arise as connecting homomorphisms in the long exact sequence associated to the exact sequence
$$
0\to \fm^n\to R\to R/\fm^n\to 0\,,
$$
with $n=j$ and $n=j-1$.
\end{remark}

\begin{chunk}\label{G1}
We state here a needed result of  Rossi and \c Sega \cite[Lemma 1.2]{RossiSega}:

 Let $\varkappa\col (P,\fp,k)\to(R,\fm,k)$ be a surjective homomorphism of local rings. Assume there exists an integer $a$ such that 
\begin{enumerate}
\item The map $\Tor_i^P(R,k)\to\Tor_i^P(R/\fm^a,k)$ induced by the natural projection $R\to R/\fm^a$ is zero for all $i > 0.$
\item The map $\Tor_i^P(\fm^{2a},k)\to\Tor_i^P(\fm^a,k)$ induced by the inclusion $\fm^{2a}\hookrightarrow \fm^a$ is zero for all $i\ge 0.$
\end{enumerate}
Then $\varkappa$ is a Golod homomorphism.
\end{chunk}

\begin{proposition}\label{nu=0}
Let $(R,\fm,k)$ be a local ring and let $j\ge 2$ be an integer. The following are equivalent:
\begin{enumerate}[\quad\rm(1)]
\item $\fm^j$ is a small submodule of $\fm^{j-1}$;
\item $\Tor_i^{\rho_j}(R/\fm^{j-1},k)$ is injective for all $i\ge 0$;
\item $\rho_j$ is Golod and $R/\fm^{j-1}$ is inert by $\rho_j$.
\end{enumerate}
If these conditions hold, then $\rho_l$ is Golod for all integers $l$ with $j\le l\le 2j-2$.
\end{proposition}
\begin{proof}
(1)$\Rightarrow$(2): 
 Let $i\ge 0$. Set $\ov{\fm}^{j-1}=\fm^{j-1}/\fm^j$.   Consider long exact sequences associated to the exact sequence $$0\to \ov{\fm}^{j-1}\to R/\fm^j\to R/\fm^{j-1}\to 0$$ and create the following commutative diagram with exact columns. 
\[
\xymatrixrowsep{1.3pc}
\xymatrixcolsep{2.6pc}
\xymatrix{
&\Tor_{i+1}^R(R/\fm^j,k)\ar@{->}[rr]\ar@{->}[d]_{\eta_{i+1}^j}&&\Tor_{i+1}^{R/\fm^j}(R/\fm^j,k)=0\ar@{->}[d]&\\
&\Tor_{i+1}^R(R/\fm^{j-1},k)\ar@{->}[rr]^{\Tor_{i+1}^{\rho_j}(R/\fm^{j-1},k)}\ar@{->}[d]^{\Delta_i}&&\Tor_{i+1}^{R/\fm^j}(R/\fm^{j-1},k)\ar@{->}[d]&\\
&\Tor_{i}^R(\ov{\fm}^{j-1},k)\ar@{->}[rr]^{\Tor_i^{\rho_j}(\ov{\fm}^{j-1},k)}&&\Tor_{i}^{R/\fm^j}(\ov{\fm}^{j-1},k)}
\]
By Remark \ref{zero}, the hypothesis that  $\Tor_i^R(\nu_j,k)=0$  implies that $\eta_{i+1}^j=0$, and thus the connecting homomorphism  $\Delta_i$ is injective. 

Levin's proof of  \cite[3.15]{Levin} shows $\Tor_i^R(\nu_j,k)=0$ for all $i$ implies $\rho_j$ is Golod. (This also follows from the last part of the proof.) In particular, the map $\rho_j$ is small by \ref{smallstuff}(1). Since $\ov{\fm}^{j-1}$ is a direct sum of copies of $k$, it follows that $\Tor_i^{\rho_j}(\ov{\fm}^{j-1},k)$ is injective. 

The bottom commutative square yields that $\Tor_{i+1}^{\rho_j}(R/\fm^{j-1},k)$ is injective. 

(2)$\Rightarrow$(1):  Assuming that $\Tor_{i+1}^{\rho_j}(R/\fm^{j-1},k)$ is injective, the top square in the commutative diagram above  gives that $\eta^j_{i+1}=0$, and thus $\Tor_i^R(\nu_j,k)=0$ by Remark \ref{zero}.

(1)$\Rightarrow$(3): As mentioned above, $\Tor_i^R(\nu_j,k)=0$ for all $i$ implies that $\rho_j$ is Golod. Since we already proved (1)$\Rightarrow$(2), we know that $\Tor_{i}^{\rho_j}(R/\fm^{j-1},k)$ is injective for all $i$. By \ref{eq}, $R/\fm^{j-1}$ is then inert by $\rho_j$.

(3)$\Rightarrow$(2): see \ref{eq}.

 Fix now $l$ such that $j\le l\le 2j-2$. We prove the last assertion of the proposition by applying \ref{G1}, with  $\varkappa=\rho_l$ and $a=j-1$. Set $\ov{R}=R/\fm^l$ and $\ov\fm =\fm /\fm^l$. Let
$$\ov{\rho}_{j-1}\col \ov{R}\to \ov{R}/\ov{\fm}^{j-1}$$
denote the canonical projection.  To satisfy the first hypothesis of \ref{G1}, we will show that the induced map
$$
\Tor_i^R(\ov{\rho}_{j-1},k): \Tor_i^R(\ov{R},k)\to \Tor_i^R(\ov{R}/\ov{\fm}^{j-1},k)
$$
 is zero for all $i>0$. Since $l\ge j$, we have $\ov{R}/\ov{\fm}^{j-1}=R/\fm^{j-1}$ and $\Tor_i^R(\ov{\rho}_{j-1},k)$ factors through 
$$\eta_i^j\col \Tor_{i}^R(R/\fm^j,k)\to\Tor_{i}^R(R/\fm^{j-1},k)\,.$$
Since $\Tor_i^R(\nu_j,k)=0$ for all $i\ge 0$ by assumption, we have that $\eta_i^j=0$ for all $i>0$  by Remark \ref{zero}. Hence $\Tor_i^R(\ov{\rho}_{j-1},k)=0$ for all $i>0$.

To satisfy the second hypothesis of \ref{G1}, we need to show that the induced map
$$
\Tor_i^R(\ov{\fm}^{2(j-1)},k)\to \Tor_i^R(\ov{\fm}^{j-1},k)
$$
induced by the inclusion $\ov{\fm}^{2(j-1)}\hookrightarrow\ov{\fm}^{j-1}$ is zero for all $i\ge 0$. In fact, this map is trivially zero since the inclusion $\fm^{2j-2}\subseteq\fm^l$ (given by the inequality $l\le 2j-2$)  implies $\ov{\fm}^{2j-2}=0$. Hence $\rho_l$ is Golod by \ref{G1}. 
\end{proof}

\begin{lemma}\label{CIinert}
Let $(R,\fm,k)$ be a local ring. If an integer $t$ satisfies $2\le t\le v(R)$, then $R/\fm^{t-1}$ is inert by  $\rho_t$. 
\end{lemma}

\begin{proof} 
We may assume that $R$ is complete. Let $R=Q/I$ be a minimal Cohen presentation, with $(Q,\fn,k)$ a regular local ring. Since $t\le v(R)$, we have $I\subseteq \fn^t$.  We can make thus the identification $R/\fm^j=Q/\fn^j$ for all $j\le t$. Let $\varkappa: Q\to R$ and $\alpha_t: Q\to Q/\fn^t$ denote the canonical projections. Since $\alpha_t=\rho_t\circ \varkappa$, \ref{2of3} shows that it suffices to prove that $R/\fm^{t-1}$ is inert by $\alpha_t$. This can be seen by applying \ref{inert3} with $J=\fn^t$, $S=Q/\fn^t$, $\overline{\fn}=\fn/\fn^t$ and $M=S/\overline{\fn}^{t-1}=Q/\fn^{t-1}$. 
\end{proof}

\begin{theorem}\label{K2}
Let $(R,\fm,k)$ be a local ring and let $t$ be an integer satisfying $2\le t\le v(R)$. The following are equivalent:
\begin{enumerate}[\quad\rm(1)]
\item $\fm^t$ is a small submodule of $\fm^{t-1}$; 
\item $\rho_t$ is small;
\item $\rho_j$ is small for all $j\ge t$;
\item $\rho_t$ is Golod;
\item $\rho_j$ is Golod for all $j$ such that $t\le j\le 2t-2$;
\item $R$ is $t$-homogeneous and the algebra $\Ext_R^*(k,k)$ is generated by $\Ext^1_R(k,k)$ and $\Ext^2_R(k,k)$.
\end{enumerate}
\end{theorem}  

\begin{proof}
The homological properties under consideration are invariant under completion.  We may assume thus $R$ is complete. Hence $R=Q/I$ with $(Q,\fn,k)$ a regular local ring and $I\subseteq \fn^t$, with $t\ge 2$.  In particular, we can make the identification $R/\fm^t=Q/\fn^t$. 

(2)$\Rightarrow$(3): This follows immediately from the definition of small homomorphim. 

(3)$\Rightarrow$(2): Clear. 

(3)$\Rightarrow$(4): Since $R/\fm^t=Q/\fn^t$ is Golod (see \ref{goloddef}), we can apply \ref{smallstuff}(2).

(4)$\Rightarrow$(2): See \ref{smallstuff}(1).

(4)$\Rightarrow$(1): By assumption $\rho_t$ is Golod. By Lemma \ref{CIinert}, $R/\fm^{t-1}$ is inert by $\rho_t$. Hence $\fm^t$ is a small submodule of $\fm^{t-1}$  by Proposition \ref{nu=0}.

(1)$\Rightarrow$(5): See Proposition \ref{nu=0}.

$(5)\Rightarrow$(4): Clear.  

$(2)\Rightarrow$(6): Assume $\rho_t$ is small, hence $\Ext_{\rho_t}(k,k)$ is a surjective homomorphism of graded algebras.  In \cite[5.9]{finiteext}, Levin shows $\Ext_{Q/\fn^t}(k,k)$ is generated by elements in degree $1$ and $2$. It follows that  $\Ext_R(k,k)$ is also generated in degrees $1$ and $2$. To see that $R$ is $t$-homogeneous, use  Remark \ref{thom}(2).

$(6)\Rightarrow$(2): Assume $R$ is $t$-homogeneous and the Yoneda algebra $\Ext_R(k,k)$ is generated by $\Ext^1_R(k,k)$ and $\Ext^2_R(k,k)$. To show that $\Ext_{\rho_t}(k,k)$ is surjective, it suffices to show that $\Ext_{\rho_t}^1(k,k)$ and $\Ext_{\rho_t}^2(k,k)$ are surjective. Since $t\ge 2$, we have that $\Ker(\rho_t)\subseteq \fm^2$, hence $\Ext_{\rho_t}^1(k,k)$ is an isomorphism, as discussed in \ref{smallstuff}. The fact that  $\Ext_{\rho_t}^2(k,k)$ is surjective is given by Remark \ref{thom}(2).
\end{proof}

We say that $R$ is a complete intersection  if the ideal $I$ in a minimal Cohen presentation $\widehat R=Q/I$ is generated by a regular sequence. For such rings, the structure of the algebra $\Ext_R(k,k)$ is known, see Sj\"odin \cite[\S 4]{Sjodin}. In particular, it is known that this algebra is generated in degrees $1$ and $2$. 

\begin{corollary}
If $R$ is a $t$-homogeneous complete intersection, then conditions (1)-(5) of the Theorem  hold. \qed
\end{corollary}

\begin{remark}
Connected $k$-algebras satisfying the condition that the Yoneda algebra is generated in degrees $1$ and $2$ are called $\mathcal K_2$ algebras by Cassidy and Shelton \cite{CassidyShelton}. Since Koszul algebras are characterized by the fact that their Yoneda algebras are generated in degree $1$, the notion of $\mathcal K_2$ algebra can be thought of as a generalization of  the notion of Koszul algebra.   
\end{remark}

\section{Compressed  Gorenstein local rings}

Compressed Gorenstein local rings  have been recently studied by Rossi and \c Sega \cite{RossiSega}; we recall below the definition given there. We consider this large class of rings as a case study for the homological properties of interest. 

\begin{bfchunk}{Compressed  Gorenstein local rings.}\label{compdef}
Let $(R, \fm, k)$ be a Gorenstein artinian local ring. The embedding dimension of $R$ is the integer $e=\rank_k(\fm/\fm^2)$, and  the socle degree of $R$ is the integer $s$ such that $\fm^s\ne 0= \fm^{s+1}$. Since $R$ is complete, a  minimal Cohen presentation of $R$ is  $R=Q/I$ with $(Q,\fn,k)$ a regular local ring and $I\subseteq \fn^2$. Set $$\varepsilon_i=\min \left\{ \binom{e-1+s-i}{e-1}, \binom{e-1+i}{e-1}\right\} \quad \text{for all $i$ with $0\le i\le s$.}$$ 

According to \cite[4.2]{RossiSega}, we have  
\begin{equation}
\label{compressed}
\lambda(R)\le \sum_{i=0}^e \varepsilon_i\,,
\end{equation}
where $\lambda(R)$ denotes the \textit{length} of $R$. We say that $R$ is a \textit{compressed} Gorenstein  local ring of socle degree $s$ and embedding dimension $e$ if $R$ has maximal length, that is, equality holds in \eqref{compressed}. 

If $R$ as above is compressed, we set 
\begin{equation}\label{t}
t=\left\lceil \frac{s+1}{2} \right\rceil\quad\text{and}\quad r=s+1-t\,,
\end{equation}
where $\lceil x \rceil$ denotes the smallest integer not less than a rational number $x$.
 
As discussed in \cite[4.2]{RossiSega}, we have $t=v(R)$. Note that if $s$ is even then $s=2t-2$ and $r=t-1$. If $s$ is odd then $s=2t-1$ and $r=t$.
\end{bfchunk}

\begin{remark}\label{degreet2}
It is shown in \cite[4.2(c)]{RossiSega} that if $R$ is a compressed Gorenstein local ring, then $\agr R$ is Gorenstein, and it is thus a compressed Gorenstein $k$-algebra. Note that compressed Gorenstein algebras can be regarded as being generic Gorenstein algebras, see the discussion in \cite[5.5]{RossiSega}. 

Let $R$ be a compressed Gorenstein local ring of socle degree $s$. When $s$ is even, the minimal free resolution of $\agr R$ over $\agr Q$ is described for example by Iarrobino in \cite[4.7]{Iarrobino}; in particular, it follows that $I^*$ is generated by homogeneous polynomials of degree $t$. According to Lemma \ref{degreet}, it follows that $R$ is $t$-homogeneous. 

When $s$ is odd, $I^*$ can be generated in degrees $t$ and $t+1$; see \cite[Proposition 3.2]{Boij}. It is conjectured in \cite[3.13]{Boij} that $I^*$ is generated  in degree $t$, and thus it is $t$-homogeneous, when $\agr R$ is generic in a stronger sense. 
\end{remark}

For the remainder of the section we use the assumptions and notation below. 

\begin{chunk}
\label{RS}
Let $(R,\fm,k)$ be a compressed Gorenstein local ring of embedding dimension $e$ and socle degree $s$, with $2\le s\ne 3$.  We consider a minimal Cohen presentation $R=Q/I$ with  $(Q,\fn,k)$ a  regular local ring and $I\subseteq \fn^2$. Since $t=v(R)$ we have $I\subseteq \fn^t$ and $I\not\subseteq \fn^{t+1}$.  Let $h\in I\cap \fn^t\setminus\fn^{t+1}$. Set $P=Q/(h)$ and $\fp=\fn/(h)$. Let $\varkappa:P\to R$ denote the canonical projection. The following properties shown in \cite{RossiSega} will be useful for our approach: 
\begin{enumerate}[\quad\rm(1)]
\item $\fm^{r+1}$ is a small submodule of $\fm^r$  (\cite[Theorem 3.3]{RossiSega}); 
\item $R/\fm^j$ is a Golod ring for $2\le j\le s$ (\cite[Proposition 6.3]{RossiSega}); 
\item $\varkappa: P\to R$ is a Golod homomorphism (\cite[Theorem 5.1]{RossiSega}). 
\item $\Po_k^R(z)\cdot d_R(z)=\Po_k^Q(z)$   (see \cite[Theorem 5.1]{RossiSega}), where
$$
d_R(z)=1-z(\Po_R^Q(z)-1)+z^{e+1}(1+z).
$$
Note that $d_R(z)$ is polynomial of degree $e+2$, since $\Po_R^Q(z)$ is a polynomial of degree $e$.
\end{enumerate} 
\end{chunk}

\begin{remark}
\label{eta}
Let $\eta\col Q\to P$ denote the canonical projection. If $M$ is an $R$-module with $\fm^{t-1}M=0$, then \ref{inert3} shows that $M$ is inert by $\varkappa\circ \eta$, since $I\subseteq \fn^t$.  It follows that $M$ is also inert by $\varkappa$, by \ref{2of3}. 

Note that the condition $\fm^{t-1}M=0$ is satisfied for $M=R/\fm^j$ with $j\le t-1$ and also for $M=\fm^j$ with $j\ge r+1$ (since $t-1+r+1=s+1$), and thus  $M$ is inert by $\varkappa$ and by $\varkappa\circ \eta$, by the above. The case $M=\fm^r$ is treated below. 
\end{remark}

\begin{lemma}
\label{kappa}
The $R$-module $\fm^r$ is inert by $\varkappa$. 
\end{lemma}

\begin{proof}
Let $i\ge 0$. Consider the commutative diagram: 
\[
\xymatrixrowsep{1.5pc}
\xymatrixcolsep{1.9pc}
\xymatrix{
\Tor_{i}^P(\fm^{r+1},k)\ar@{->}[rr]^{\Tor_i^P(\nu_{r+1},k)}\ar@{->}[d]&&\Tor_i^P(\fm^r,k)\ar@{->}[d]^{\beta}\ar@{->}[r]^{\alpha\quad}&\Tor_i^P(\fm^r/\fm^{r+1},k)\ar@{->}[d]^{\gamma}\\
\Tor_{i}^R(\fm^{r+1},k)\ar@{->}[rr]&&\Tor_i^R(\fm^r,k)\ar@{->}[r]&\Tor_i^R(\fm^r/\fm^{r+1},k)
}
\]
where $\beta=\Tor_i^{\varkappa}(\fm^r,k)$ and $\gamma=\Tor_i^{\varkappa}(\fm^r/\fm^{r+1},k)$. Since $\Tor_i^P(\nu_{r+1},k)=0$ by \ref{RS}(1), $\alpha$ is injective. Since $\varkappa$ is Golod, it is in particular small by \ref{smallstuff}(1), and it follows that $\gamma$ is injective, since $\fm^r/\fm^{r+1}$ is a direct sum of copies of $k$. The commutative square on the right shows that $\beta$ is injective as well, hence $\fm^r$ is inert by $\varkappa$ by \ref{eq}. 
\end{proof}

We  now prove Theorem 2 in the introduction. We restate it below, with some more detail in part (1).  

\begin{theorem}\label{main}
Let $2\le s\ne 3$ and let  $R$ be a compressed Gorenstein local ring of socle degree $s$.  Let $n\ge 1$. The following hold: 
\begin{enumerate}[\quad\rm(1)]
\item  $\fm^n$ is a small submodule of $\fm^{n-1}$  if and only if $n>s$ or $n=r+1$. 
Furthermore,  if $n\ne r+1$ and $n\le s$, then $\Tor_i^R(\nu_n,k)\ne 0$ for infinitely many values of $i$. 
\item $\rho_n\col R\to R/\fm^n$ is Golod if and only if $n\ge r+1$.
\end{enumerate}
\end{theorem}

\begin{corollary}
\label{cor}
With $R$ as in the theorem, the following hold: 
\begin{enumerate}[\quad\rm(1)]
\item If $s$ is even, then $\Ext_R(k,k)$ is generated by $\Ext^1_R(k,k)$ and $\Ext^2_R(k,k)$.
\item  If $s$ is odd  and $R$ is $t$-homogeneous, then $\Ext_R(k,k)$ is not generated by $\Ext^1_R(k,k)$ and $\Ext^2_R(k,k)$. 
\end{enumerate}
\end{corollary}
\begin{proof}
If $s$ is even, then $r=t-1$ and Theorem \ref{main}(1) gives that $\fm^t$ is a small submodule of $\fm^{t-1}$. 
If $s$ is odd, then $r=t$, and Theorem \ref{main}(1) gives that $\Tor_i^R(\nu_t,k)\ne 0$ for infinitely many values of $i$, hence $\fm^t$ is not a small submodule of $\fm^{t-1}$. 
Both conclusions follow then from Theorem \ref{K2}. 
\end{proof}

\begin{proof}[Proof of Theorem {\rm\ref{main}}]
Assuming that (1) is proved, we prove (2) as follows. 

Since $\fm^{r+1}$ is a small submodule of $\fm^r$, we know that that $\rho_{r+1}$ is a Golod homomorphism, and furthermore a small homomorphism (see Section 2).  Let $n\ge r+1$. The homomorphism $\rho_n$ factors thorugh $\rho_{r+1}$, and it is thus small as well. Since $R/\fm^j$ is Golod by \ref{RS}(2), it follows by (2) in \ref{smallstuff} that $\rho_j$ is Golod. If $n\le r$, then we also have $n\le t$, since $r=t$ or $r=t-1$. 
In view of Theorem \ref{K2}, the fact that $\Tor_*^R(\nu_n,k)\ne 0$ in this case  implies that $\rho_n$ is not Golod. 

We now prove (1).  Let $j\ge 0$. We use the notation introduced in \ref{A=0}, noting that $\nu_{\fm^j}=\nu_{j+1}$. We have 
$$
\Tor_*^R(\nu_{j+1},k)=0 \iff T^R_{\fm^j}(z)=0 
$$
and $\Tor_i^R(\nu_{j+1},k)\ne 0$ for infinitely many $i$ if and only if $T^R_{\fm^j}(z)\notin \mathbb Z[z]$.

The conclusion will be established through a concrete computation of  $T^R_{\fm^{j}}(z)$. 

Using \eqref{poincare}  we have: 
\begin{equation}
\label{S}
\Po_{\fm^{j}}^S(z)-a_j\Po_k^S(z)+z\Po_{\fm^{j+1}}^S(z)=(1+z)T^S_{\fm^{j}}(z)\\
\end{equation}
where $a_j=\rank_k(\fm^{j}/\fm^{j+1})$ and $S=R$ or $S=P$ or $S=Q$. There are four distinct cases to be considered: 

\begin{Case1}
Assume $j=r$. Recall that $\fm^{r+1}$ and $\fm^{r}$ and $k$ are all inert by $\varkappa$, by \ref{kappa} and \ref{eta}. Using the definition of inertness for each of these modules, an application of the formula \eqref{S} for $j=r$, with $S=R$ and then with  $S=P$, gives: 
$$
T^R_{\fm^r}(z)=T^P_{\fm^r}(z)\cdot \frac{\Po^R_k(z)}{\Po^P_k(z)}\,.
$$
Since we know that $\Tor_*^P(\nu_{r+1},k)=0$, see \ref{RS}(1),  we have that $T^P_{\fm^r}(z)=0$, hence $T^R_{\fm^r}(z)=0$ and thus $\Tor_*^R(\nu_{r+1},k)=0$. 
\end{Case1}

\begin{Case2}
Assume $r+1\le j<s$.  We know that $\fm^{j+1}$, $\fm^{j}$ and $k$ are all inert by $\varkappa\circ\eta\col Q\to R$ by \ref{eta}. Proceeding as above, we obtain: 
$$
T^R_{\fm^j}(z)=T^Q_{\fm^j}(z)\cdot \frac{\Po^R_k(z)}{\Po^Q_k(z)}=\frac{T^Q_{\fm^j}(z)}{d_R(z)}
$$
where the second equality is obtained using \ref{RS}(4).

In \cite[Lemma 4.4]{RossiSega} it is proved that the map $\Tor_i^Q(\nu_{r+1},k)$ is zero for all $i\ne e$ and is bijective for $i=e$. The argument given in the proof there, with a minor adjustment, shows that the following more general statement holds: For any $j$ with  $r\le j\le s$, the map $\Tor_i^Q(\nu_{j+1},k)$ is zero for all $i\ne e$ and is bijective for $i=e$. 

Note that $\Tor_e^Q(\fm^j,k)\cong \Soc(\fm^j)$, the socle of $\fm^j$. Since $R$ is Gorenstein, $\rank_k \Soc(\fm^j)=\rank_k\Soc(R)=1$. It follows that 
$$T^Q_{\fm^j}(z)=z^e$$
and thus $T^R_{\fm^j}(z)$ is a quotient of a polynomial of degree $e$ by a polynomial of degree $e+2$. We conclude that  $T^R_{\fm^j}(z)$ is not a polynomial and thus $\Tor_i^R(\nu_{j+1},k)\ne 0$ for infinitely many $i$. (On the other hand, note that $T^R_{\fm^j}(z)$ is a multiple of $z^e$ in $\mathbb Z[[z]]$, and this implies that  $\Tor_i^R(\nu_{j+1},k)= 0$ for all $i<e$.)
\end{Case2}

\begin{Case3}
Assume $j=t-1$ and $j<r$. Since $r=t-1$ when $s$ is even, this case can happen only when $s$ is odd. In this case, one has $r=t$, hence $j=r-1$ as well. 
In particular, $j+1=r$ and we use the already established fact that $T^R_{\fm^r}(z)=0$ in the second line below, in order to replace $\Po_{\fm^{j+1}}^R(z)$. 
\begin{align*}
z(1+z)T^R_{\fm^{j}}(z)&=z\Po_{\fm^{j}}^R(z)-a_{j}z\Po_k^R(z)+z^2\Po_{\fm^{j+1}}^R(z)\\
&=(\Po^R_{R/\fm^{j}}(z)-1)-a_{j}z\Po_k^R(z)+z^2\big(a_{j+1}\Po^R_k(z)-z\Po^R_{\fm^{j+2}}(z)\big)\\
&=\left(\Po^Q_{R/\fm^{j}}(z)-(a_{j}z-a_{j+1}z^2)\Po_k^Q(z)-z^3\Po^Q_{\fm^{j+2}}(z)\right)\cdot  \frac{\Po^R_k(z)}{\Po^Q_k(z)} -1
\end{align*}
For the last equality, we have used the definition of inertness and the fact that the $R$-modules $R/\fm^j$, $\fm^{j+2}$ and $k$ are all inert by $\varkappa\circ\eta$; this can be seen using again \ref{eta}, since $j=t-1$ and $j+2=r+1$. 
Using  \ref{RS}(4), we have
$$
T^R_{\fm^{j}}(z)=\frac{\Po^Q_{R/\fm^{j}}(z)-(a_{j}z-a_{j+1}z^2)\Po_k^Q(z)-z^3\Po^Q_{\fm^{j+2}}(z)-d_R(z)}{z(z+1)d_R(z)}
$$
Since $d_R(z)$ has degree $e+2$ and $\Po^Q_{\fm^{j+2}}(z)$ is a polynomial of degree $e$ (note that $\fm^{j+2}=\fm^{t+1}\ne 0$), the outcome of this computation is that $T^R_{\fm^{j}}(z)$ is a quotient of a polynomial of degree $e+3$ by a polynomial of degree $e+4$. Again, it is clear that $T^R_{\fm^{j}}(z)$ cannot be a polynomial. 
\end{Case3}

\begin{Case4}
Assume $j\le t-2$.
We have: 
\begin{align*}
z(1+z)T^R_{\fm^j}(z)&=z\Po_{\fm^{j}}^R(z)-a_jz\Po_k^R(z)+z^2\Po_{\fm^{j+1}}^R(z)\\
&=(\Po^R_{R/\fm^j}(z)-1)-a_jz\Po_k^R(z)+z\big(\Po^R_{R/\fm^{j+1}}(z)-1\big)\\
&=\left(\Po^Q_{R/\fm^j}(z)-a_jz\Po_k^Q(z)+z\Po^Q_{R/\fm^{j+1}}(z)\right)\cdot \frac{\Po^R_k(z)}{\Po^Q_k(z)}-1-z
\end{align*}
where the third equality is due to the fact that $R/\fm^j$, $R/\fm^{j+1}$ and $k$ are all inert by $\varkappa\circ\eta$, in view of  \ref{eta}. 
Using   \ref{RS}(4) as above, one sees that $T^R_{\fm^j}(z)$ can be written as a qotient of a polynomial of degree $e+3$ by a polynomial of degree $e+4$, and thus it is not a polynomial. 
\end{Case4}
Since $\Tor_*^R(\nu_j,k)=0$ is clearly zero when $j>s$, we exhausted all cases for $j$. 
\end{proof}

\section{Acknowledgement}
We would like to thank Frank Moore for a useful discussion about $\mathcal K_2$ algebras, and Luchezar Avramov and the referee for suggestions regarding the exposition.

\end{document}